\documentclass[11pt,reqno]{amsart}

\usepackage{mlmodern}
\usepackage{amsmath, amssymb, amsthm, mathtools}
\usepackage{microtype}
\usepackage{tikz}
\usepackage{hyperref}
\hypersetup{colorlinks=true, linkcolor=black, citecolor=black, urlcolor=blue}

\theoremstyle{plain}
\newtheorem{theorem}{Theorem}[section]
\newtheorem{lemma}[theorem]{Lemma}
\newtheorem{proposition}[theorem]{Proposition}
\newtheorem{corollary}[theorem]{Corollary}

\theoremstyle{definition}
\newtheorem{definition}[theorem]{Definition}

\theoremstyle{remark}
\newtheorem{remark}[theorem]{Remark}
\newtheorem{question}[theorem]{Question}

\newcommand{\N}{\mathbb{N}}
\newcommand{\Z}{\mathbb{Z}}
\newcommand{\bN}{\beta\mathbb{N}}
\newcommand{\Nstar}{\mathbb{N}^{*}}
\newcommand{\FS}{\mathrm{FS}}
\newcommand{\E}{E(\mathbb{N}^{*})}
\newcommand{\K}{K(\beta\mathbb{N})}

\title[Explicit witnesses at every gap]{Explicit Witnesses at Every Gap\\
of the Depth Filtration of $\bN$}

\author{Carl Aza}
\email{carljuanaza@gmail.com}


\date{\today}

\begin{document}

\begin{abstract}
Let $\Sigma_{1} = \Nstar$ and $\Sigma_{k+1} = \overline{\Nstar + \Sigma_{k}}$ be
the cumulative depth filtration of $\bN$, the analogue for $(\N,+)$ of a chain
of closed ideals that Protasov and Protasova studied for discrete groups, where
strict descent follows from a theorem of Lutsenko and Protasov. For every $k$ we
give an explicit set whose closure meets $\Sigma_{k}$ but not $\Sigma_{k+1}$. Fix
the doubly exponential sequence $e_{n} = 2^{2^{n}}$, partition it into $k$
subsequences $E_{0}, \dots, E_{k-1}$ by the residue of the index modulo $k$, and
set $A_{k} = E_{0} + \cdots + E_{k-1}$. We prove that any sum $q_{0} + \cdots +
q_{k-1}$ of free ultrafilters with $E_{t} \in q_{t}$ lies in $\Sigma_{k}
\setminus \Sigma_{k+1}$. The engine is a master lemma, proved by induction on
$j$: if a sum $F_{1} + \cdots + F_{j}$ of subsequences of $\{e_{n}\}$ with
pairwise disjoint index sets belongs to a free ultrafilter $s$, then $s \notin
\Sigma_{j+1}$. The proof rests on a single rigidity of the doubly exponential
sequence: a fixed difference forces the largest index in any
shift-intersection, once it is large, to cancel within its own subsequence, which makes every
shift-intersection descend by at least one level. The same witnesses lie in the
gaps of the pure filtration.
\end{abstract}

\maketitle

\section{Introduction}

The Stone--\v{C}ech compactification $\bN$ of $(\N, +)$ is a compact
right topological semigroup whose algebra encodes much additive Ramsey theory
\cite{HindmanStrauss1998}. Write $\Nstar = \bN \setminus \N$ for the free
ultrafilters. For a discrete group $G$, Protasov and Protasova
\cite[\S 4]{ProtasovProtasova2017} studied the chain
\[
  I_{G,0} = G^{*}, \qquad I_{G,n+1} = \overline{G^{*}I_{G,n}}.
\] Its analogue for $(\N,+)$ is the \emph{cumulative
depth filtration}
\[
  \Sigma_{1} = \Nstar,
  \qquad
  \Sigma_{k+1} = \overline{\Nstar + \Sigma_{k}}
  \quad (k \geq 1),
\]
a descending chain of closed two-sided ideals of $\Nstar$. (Our $\Sigma_{k}$
corresponds to $I_{G, k-1}$, an index shift we adopt so that $\Sigma_{k}$ is built
from $k$-fold sums.) Its gaps $\Sigma_{k} \setminus \Sigma_{k+1}$ measure
additive depth: roughly, a point of the $k$-th gap is built from iterated sums
to depth $k$ but no further.

Protasov and Protasova proved their chain strictly decreasing for every infinite group
\cite[Theorem~4.2]{ProtasovProtasova2017}, by applying a theorem of Lutsenko and
Protasov on relatively sparse sets \cite[Theorem~4(2)]{LutsenkoProtasov2011}. The
Lutsenko--Protasov proof builds its witnessing sets from inductively chosen sequences; it uses a thin
set at the first level and writes out sets $Q_{1}$ and $Q_{2}$, ordered sums of two
and three sequences, for the next two levels. The group results do not formally
give the statement for $\N$, which is not a group, and we have not found it
recorded for $k \geq 2$. For $\N$ the first gap is classical: any free
ultrafilter on a Sidon set lies in $\Sigma_{1} \setminus \Sigma_{2}$
(Lemma~\ref{lem:firstgap}; see \cite[Exercise~4.1.7]{HindmanStrauss1998} and
\cite{FilaliLutsenkoProtasov2009}; cf.\ \cite[Theorem~8.1]{vanDouwen1991}). Carried out inside $\N$, the
Lutsenko--Protasov construction would presumably also give members of the next
two gaps, though not in closed form. We treat all levels at once,
producing a member of $\Sigma_{k} \setminus \Sigma_{k+1}$ for every $k$ by a
uniform construction from closed-form sets.

\begin{definition}\label{def:partition}
Let $e_{n} = 2^{2^{n}}$ for $n \geq 1$. For $k \geq 1$ and $t \in \{0, \dots,
k-1\}$ set
\[
  E_{t}^{(k)} = \{ e_{n} : n \equiv t \!\!\pmod k \},
  \qquad
  A_{k} = E_{0}^{(k)} + \cdots + E_{k-1}^{(k)}.
\]
\end{definition}

\begin{theorem}\label{thm:main}
Let $k \geq 1$, and let $q_{0}, \dots, q_{k-1} \in \Nstar$ be free ultrafilters
with $E_{t}^{(k)} \in q_{t}$ for each $t$. Then
\[
  p := q_{0} + \cdots + q_{k-1} \in \Sigma_{k} \setminus \Sigma_{k+1}.
\]
\end{theorem}

The witnesses are ultrafilters, so their existence uses the ultrafilter lemma, a weak form of the axiom of choice;
what is explicit is the set $A_{k}$, whose closure meets $\Sigma_{k}$ but not
$\Sigma_{k+1}$. The lower bound is immediate, as $p$ is a $k$-fold sum of free ultrafilters and
$\Nstar + \cdots + \Nstar \subseteq \Sigma_{k}$ (Proposition~\ref{prop:two-sided-k}). The upper bound is a consequence
of a single inductive lemma. Call a finite collection $F_{1}, \dots, F_{j}$ of
subsets of $\Z$ a \emph{separated system} if each $F_{i} = \{ e_{n} + c_{i} : n
\in S_{i} \}$ is an integer translate of an infinite subsequence of $\{e_{n}\}$
and the index sets $S_{i}$ are pairwise disjoint. (Allowing translates costs
nothing and is needed for the induction, since a shift-intersection of a
separated-system sum is covered by translates of sub-sums; see
Section~\ref{sec:shift}.)
The sets $E_{0}^{(k)}, \dots, E_{k-1}^{(k)}$ form a separated system of size $k$
with all $c_{i} = 0$. The heart of the paper is:

\begin{theorem}[Master Lemma]\label{thm:master}
Let $F_{1}, \dots, F_{j}$ be a separated system $(j \geq 1)$. If $(F_{1} + \cdots
+ F_{j}) \cap \N \in s$ for a free ultrafilter $s$, then $s \notin \Sigma_{j+1}$.
\end{theorem}

Theorem~\ref{thm:main} follows at once: $A_{k} \in p$ by
Lemma~\ref{lem:summembers}, and $A_{k}$ is a sum over a separated system of size
$k$, so Theorem~\ref{thm:master} gives $p \notin \Sigma_{k+1}$, while $p \in
\Sigma_{k}$ trivially.

The mechanism never uses that an ultrafilter is an actual iterated sum, only that
it lies in a given level of the filtration; this is what allows the induction to
run inside the cumulative filtration, where the levels are closures, rather than
in the formally simpler filtration of pure $k$-fold-sum closures. As a
consequence the same points also occupy the gaps of that \emph{pure}
filtration $\Pi_{k} = \overline{\Nstar + \cdots + \Nstar}$ ($k$ summands), since
$\Pi_{k} \subseteq \Sigma_{k}$; see Section~\ref{sec:remarks}.

Section~\ref{sec:prelim} fixes notation and the rigidity facts.
Section~\ref{sec:shift} proves the shift-intersection lemma: every shift-intersection of a separated-system sum is covered, up to a finite
set, by translates of sub-sums of strictly fewer pieces. Section~\ref{sec:master} proves the Master Lemma by induction,
Section~\ref{sec:remarks} comments on the pure filtration, the rigidity and the
group $\Z$, and Section~\ref{sec:bottom} records what is known about the bottom
$\Sigma_{\infty} = \bigcap_{k}\Sigma_{k}$ of the filtration.

\subsection*{Related work}
The combinatorial tool underlying our argument, controlling an ultrafilter sum
through the leftward shifts $-x + B$ of its members and the finite-intersection
structure of $\{-a + B\}$, belongs to a well-developed circle of ideas around
\emph{finite embeddability} of sets and ultrafilters, developed by Di Nasso
\cite{DiNasso2014}, Blass and Di Nasso \cite{BlassDiNasso2015}, and others, and
to ultrafilter methods in combinatorial number theory more broadly (see for
example Beiglb\"ock \cite{Beiglbock2011}). In particular, Blass and Di Nasso
\cite[Theorem~11]{BlassDiNasso2015} characterize the closure of a principal right ideal
$\{\mathcal{U} + \mathcal{W} : \mathcal{W} \in \bN\}$ in these terms, and their
Example~15 already features the set $\{2^{m}\}$ in a shift construction. That
work analyzes a different object, namely the single-generator cones of the
finite embeddability order, and does not address the iterated depth filtration or its
gaps. The same device underlies the sparseness and relative-sparseness methods
of Filali, Lutsenko and Protasov \cite{FilaliLutsenkoProtasov2009} and of
Lutsenko and Protasov \cite{LutsenkoProtasov2011}, whose witnesses for
groups are built inductively rather than given in closed form. We therefore make
no claim of novelty for the shift-intersection technique itself. Our
contribution is to apply this established toolkit to the depth filtration of
$\N$, giving at every gap an explicit set whose closure meets that level but not
the next.

\section{Preliminaries}\label{sec:prelim}

We adopt the conventions of Hindman and Strauss~\cite{HindmanStrauss1998}.  The set
$\N = \{1, 2, 3, \ldots\}$ is given the discrete topology, and $\bN$ denotes the
set of ultrafilters on $\N$.  Each $n \in \N$ is identified with the principal
ultrafilter $\dot{n} = \{A \subseteq \N : n \in A\}$, so that $\N \subseteq \bN$.
The set $\Nstar := \bN \setminus \N$ is the set of free (non-principal) ultrafilters.
For $A \subseteq \N$, the set $\overline{A} = \{p \in \bN : A \in p\}$ is clopen
in $\bN$, and $\{\overline{A} : A \subseteq \N\}$ is a basis for the topology.

The operation $+$ on $\N$ extends to $\bN$ via
\[
  A \in p + q \iff \{n \in \N : -n + A \in q\} \in p,
\]
where $-n + A = \{m \in \N : n + m \in A\}$.  Under this operation $(\bN, +)$ is a
compact right topological semigroup: for each fixed $q \in \bN$, the map $\rho_{q}
: p \mapsto p + q$ is continuous.  Left multiplication $\lambda_{r} : q \mapsto r +
q$ is in general not continuous when $r \in \Nstar$.  The set $\Nstar$ is a closed
sub-semigroup of $\bN$.

By Ellis's lemma, every compact Hausdorff right topological semigroup has an idempotent
\cite[Theorem~2.5]{HindmanStrauss1998}; we
write $\E = \{e \in \Nstar : e + e = e\}$ for the set of idempotents in $\Nstar$.
The semigroup $\bN$ has a smallest two-sided ideal $\K$, which is the union of all
minimal right ideals (equivalently, all minimal left ideals)
\cite[Theorem~2.8]{HindmanStrauss1998}.  Both $\E$ and $\K$
are nonempty subsets of $\Nstar$, since $\Nstar$ is an ideal of $\bN$
\cite[Corollary~4.33]{HindmanStrauss1998}.

\subsection{The depth filtration and its basic properties}

The following is the analogue for $(\N,+)$ of a definition of Protasov and
Protasova for groups \cite[\S 4]{ProtasovProtasova2017}.

\begin{definition}\label{def:filtration}
Define a descending sequence of subsets of $\bN$ by
\[
  \Sigma_{0} := \bN, \qquad \Sigma_{1} := \Nstar,
\]
and inductively for $k \geq 1$,
\[
  \Sigma_{k+1} := \overline{\Nstar + \Sigma_{k}},
\]
where the closure is taken in $\bN$.  Set $\Sigma_{\infty} := \bigcap_{k=0}^{\infty}
\Sigma_{k}$.
\end{definition}

In the notation of~\cite[Section~4]{ProtasovProtasova2017}, $\Sigma_{k}$
corresponds to the ideal $I_{G,\,k-1}$, where $I_{G,0} := G^{*}$ and $I_{G,n+1}
:= \overline{G^{*}I_{G,n}}$ for a discrete group $G$ (written multiplicatively).
Since $\N$ is not a group, we record the structural facts we need with
self-contained proofs.

\begin{remark}[notation]\label{rem:notation}
Since only the additive structure is considered here, we write $\Sigma_{k}$
without decoration.  When the additive and multiplicative filtrations of $\bN$
are treated together one writes $\Sigma_{k}^{+}$ and $\Sigma_{k}^{\times}$
respectively; in that notation every $\Sigma_{k}$ below is $\Sigma_{k}^{+}$.
We also record that $\Nstar$ denotes the Stone--\v{C}ech remainder
$\bN \setminus \N$, with the star written on the \emph{right}; this is not to
be confused with the nonstandard extension ${}^{*}\N$ of the natural numbers,
where the star is written on the left.
\end{remark}

Figure~\ref{fig:nested} illustrates Definition~\ref{def:filtration}.  The
levels are nested regions: the outermost is all of $\bN$, each step inward is a
more decomposable stratum, and the common center is $\Sigma_{\infty}$.  The
marked points record where the objects of this paper sit, in particular
witnesses of Theorem~\ref{thm:main} in the first two gaps, starting with the
outermost \emph{free} ring $\Sigma_{1} \setminus \Sigma_{2}$.

\begin{figure}[ht]
\centering
\begin{tikzpicture}[scale=0.95]
  \draw[black!55, fill=black!2]  (0,0) ellipse (5.4cm and 3.3cm);
  \draw[black!55, fill=black!5]  (0,0) ellipse (4.5cm and 2.75cm);
  \draw[black!55, fill=black!9]  (0,0) ellipse (3.5cm and 2.14cm);
  \draw[black!55, fill=black!13] (0,0) ellipse (2.55cm and 1.56cm);
  \draw[black!55, fill=black!20] (0,0) ellipse (1.5cm and 0.92cm);

  \node at (0, 2.98) {\footnotesize $\Sigma_{0} = \bN$};
  \node at (0, 2.42) {\footnotesize $\Sigma_{1} = \Nstar$};
  \node at (0, 1.83) {\footnotesize $\Sigma_{2}$};
  \node at (0, 1.25) {\footnotesize $\Sigma_{3}$};
  \node at (0, -0.02) {\footnotesize $\Sigma_{\infty}$};

  % principal point: outermost shell only
  \filldraw[black] (-4.95,0.0) circle (1.5pt);
  \node[left=1pt, align=right, fill=white, inner sep=1pt] at (-4.95,0.0)
    {\scriptsize $\dot n$\\[-2pt]\scriptsize $\Sigma_{0}\!\setminus\!\Sigma_{1}$};

  % the paper's witness: top gap
  \filldraw[black] (4.02,0.0) circle (1.7pt);
  \node[right=1pt, align=left, fill=white, inner sep=1pt] at (4.08,0.0)
    {\scriptsize witness on $\{2^{2^{n}}\}$\\[-2pt]\scriptsize $\Sigma_{1}\!\setminus\!\Sigma_{2}$};

  % a witness in the second gap
  \filldraw[black] (-3.02,0.0) circle (1.5pt);
  \node[below=2pt, align=center, fill=white, inner sep=1pt] at (-3.02,-0.12)
    {\scriptsize $q_{0}+q_{1}$\\[-2pt]\scriptsize $\Sigma_{2}\!\setminus\!\Sigma_{3}$};

  % idempotent at the center
  \filldraw[black] (0.0,-0.55) circle (1.7pt);
  \node[below=1pt] at (0.0,-0.66) {\scriptsize $e = e + e$};
  \node at (0, 0.56) {\footnotesize $\vdots$};
\end{tikzpicture}
\caption{The levels of Definition~\ref{def:filtration} as nested regions.  Each level is
nested inside the previous one.  A principal ultrafilter $\dot n$ lies in the
outermost shell only; an idempotent $e = e+e$ lies in every level, hence in
$\Sigma_{\infty}$; a sum of two free ultrafilters lies in $\Sigma_{2}$.  By
Theorem~\ref{thm:main}, a free ultrafilter containing $E_{0}^{(1)} = \{2^{2^{n}}\}$
occupies the outermost free ring $\Sigma_{1} \setminus \Sigma_{2}$, and $q_{0} +
q_{1}$ with $E_{t}^{(2)} \in q_{t}$ lies in $\Sigma_{2} \setminus \Sigma_{3}$.}
\label{fig:nested}
\end{figure}
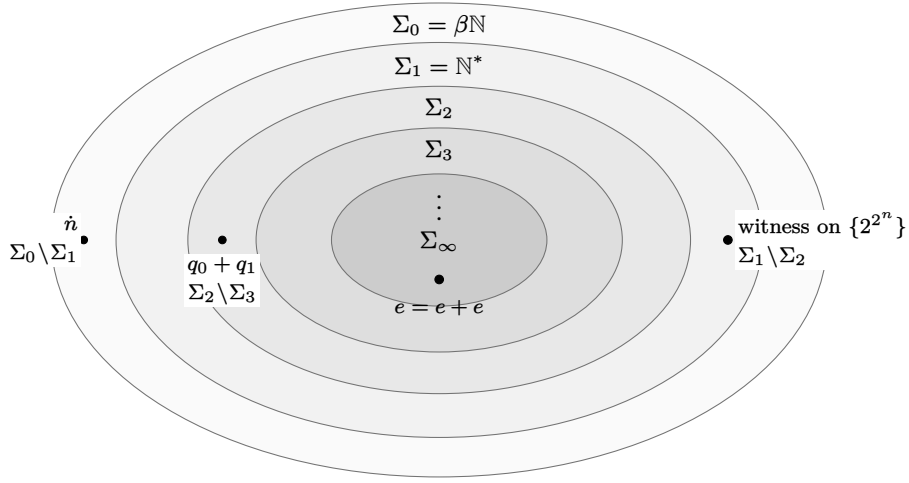

\begin{lemma}\label{lem:descending}
The sequence $(\Sigma_{k})_{k \geq 0}$ is descending: $\Sigma_{k+1} \subseteq
\Sigma_{k}$ for every $k \geq 0$.
\end{lemma}

\begin{proof}
By induction on $k$.  For $k = 0$: $\Sigma_{1} = \Nstar \subseteq \bN =
\Sigma_{0}$.  For $k = 1$: $\Nstar + \Nstar \subseteq \Nstar$ since $\Nstar$ is
a sub-semigroup of $\bN$, and $\Nstar$ is closed, so $\Sigma_{2} =
\overline{\Nstar + \Nstar} \subseteq \overline{\Nstar} = \Nstar = \Sigma_{1}$.
For the inductive step, assume $\Sigma_{k} \subseteq \Sigma_{k-1}$.  Then
$\Nstar + \Sigma_{k} \subseteq \Nstar + \Sigma_{k-1}$, and taking closures
gives $\Sigma_{k+1} \subseteq \Sigma_{k}$.
\end{proof}

\begin{lemma}\label{lem:closed}
For every $k \geq 0$, the set $\Sigma_{k}$ is closed in $\bN$.
\end{lemma}

\begin{proof}
$\Sigma_{0}$ and $\Sigma_{1}$ are closed by inspection.  For $k \geq 1$,
$\Sigma_{k+1}$ is closed by definition.
\end{proof}

\begin{lemma}\label{lem:right-ideal}
For every $k \geq 1$, $\Sigma_{k}$ is a right ideal of $\Nstar$: $\Sigma_{k} +
\Nstar \subseteq \Sigma_{k}$.
\end{lemma}

\begin{proof}
By induction on $k$.  The base case $k=1$ is trivial since $\Sigma_{1} = \Nstar$.
Assume $\Sigma_{k}$ is a right ideal.  Let $s \in \Sigma_{k+1}$ and $q \in
\Nstar$.  Choose a net $(s_{\alpha})$ in $\Nstar + \Sigma_{k}$ with $s_{\alpha}
\to s$, writing $s_{\alpha} = r_{\alpha} + p_{\alpha}$ with $r_{\alpha} \in
\Nstar$ and $p_{\alpha} \in \Sigma_{k}$.  By right-continuity,
$s_{\alpha} + q \to s + q$.  By associativity,
\(
  s_{\alpha} + q = r_{\alpha} + (p_{\alpha} + q),
\)
and $p_{\alpha} + q \in \Sigma_{k}$ by the inductive hypothesis.  Hence $s + q
\in \overline{\Nstar + \Sigma_{k}} = \Sigma_{k+1}$.
\end{proof}

\begin{proposition}\label{prop:two-sided-k}
For every $k \geq 1$, $\Sigma_{k}$ is a closed two-sided ideal of $\Nstar$.
Moreover $u_{1} + \cdots + u_{k} \in \Sigma_{k}$ for all $k \geq 1$ and $u_{1},
\dots, u_{k} \in \Nstar$.
\end{proposition}

\begin{proof}
Closedness is Lemma~\ref{lem:closed}; the right-ideal property is
Lemma~\ref{lem:right-ideal}.  For the left ideal property: by
Lemma~\ref{lem:descending}, $\Sigma_{k} \supseteq \Sigma_{k+1} =
\overline{\Nstar + \Sigma_{k}} \supseteq \Nstar + \Sigma_{k}$.  For the last
claim induct on $k$: the case $k = 1$ is $\Sigma_{1} = \Nstar$, and $u_{1} +
(u_{2} + \cdots + u_{k}) \in \Nstar + \Sigma_{k-1} \subseteq \Sigma_{k}$.
\end{proof}

\begin{proposition}\label{prop:two-sided}
$\Sigma_{\infty}$ is a closed two-sided ideal of $\Nstar$, and $\overline{\K} \cup
\overline{\E} \subseteq \Sigma_{\infty}$.  Consequently $\overline{\K} \subseteq
\Sigma_{k}$ for every $k$.
\end{proposition}

\begin{proof}
The ideal property follows from Proposition~\ref{prop:two-sided-k} by
intersection.  Every idempotent $e \in \Nstar$ satisfies $e = e + e$, so by
induction $e \in \Sigma_{k}$ for all $k$: if $e \in \Sigma_{k}$ then $e + e \in
\Nstar + \Sigma_{k} \subseteq \overline{\Nstar + \Sigma_{k}} = \Sigma_{k+1}$.
This gives $\overline{\E} \subseteq \Sigma_{\infty}$ by closedness.  (This is the
argument of \cite[Theorem~4.2(ii)]{ProtasovProtasova2017} for groups.)

For $\overline{\K} \subseteq \Sigma_{\infty}$: for $n \in \N$ and $p \in \bN$ we
have $p + n = n + p$, since each contains $A$ exactly when $-n + A \in p$; and
$\lambda_{n} \colon p \mapsto n + p$ is continuous, since
$\lambda_{n}^{-1}(\overline{A}) = \overline{-n + A}$.  As $n + (\Nstar +
\Sigma_{k}) = (n + \Nstar) + \Sigma_{k} \subseteq \Nstar + \Sigma_{k}$, because $n +
\Nstar \subseteq \Nstar$ \cite[Corollary~4.33]{HindmanStrauss1998},
continuity gives $n + \Sigma_{k+1} \subseteq \Sigma_{k+1}$; also $n + \Sigma_{1}
\subseteq \Nstar$.  Together with
$\Sigma_{k} + n = n + \Sigma_{k}$, Lemma~\ref{lem:right-ideal} and
Proposition~\ref{prop:two-sided-k}, and as $\bN = \N \cup \Nstar$, each
$\Sigma_{k}$, and so $\Sigma_{\infty}$, is a two-sided ideal of $\bN$, and
therefore contains the smallest ideal $\K$.  Closedness gives $\overline{\K}
\subseteq \Sigma_{\infty}$.
\end{proof}

Free ultrafilters contain only infinite sets. We record three more facts and
the rigidity of $\{e_{n}\}$.

\begin{lemma}[Cumulative membership criterion]\label{lem:membership}
For $j \geq 1$ and $p \in \bN$, if $p \in \Sigma_{j+1}$ then every $A \in p$
satisfies $A \in r + v$ for some $r \in \Nstar$ and $v \in \Sigma_{j}$.
\end{lemma}

\begin{proof}
By definition $\Sigma_{j+1} = \overline{\Nstar + \Sigma_{j}}$. Let $A \in p$;
then $\overline{A}$ is a clopen neighborhood of $p$, and since $p$ lies in the
closure of $\Nstar + \Sigma_{j}$ this neighborhood meets it: there is $w \in
\Nstar + \Sigma_{j}$ with $A \in w$. Writing $w = r + v$ with $r \in \Nstar$ and
$v \in \Sigma_{j}$ gives the claim.
\end{proof}

\begin{lemma}[Sums of members, {\cite[Theorem~4.15 and Exercise~4.1.6]{HindmanStrauss1998}}]\label{lem:summembers}
If $u_{1}, \dots, u_{k} \in \Nstar$ and $B_{i} \in u_{i}$, then $B_{1} + \cdots +
B_{k} \in u_{1} + \cdots + u_{k}$.
\end{lemma}

\begin{lemma}[First-gap mechanism]\label{lem:firstgap}
Call $S \subseteq \Z$ a \emph{Sidon set} if every nonzero integer has at most
one representation as a difference of two elements of $S$. Let $S \subseteq \N$
be an infinite Sidon set. If $r \in \Nstar$ with $S \in r$, then $r \notin \Sigma_{2}$, and hence
$r \notin \Sigma_{k}$ for every $k \geq 2$.
\end{lemma}

\begin{proof}
If $S \in u + w$ with $u, w \in \Nstar$, then $T = \{n : -n + S \in w\} \in
u$ is infinite, and for distinct $n, n' \in T$ the set $(-n + S) \cap (-n' + S)
\in w$ is infinite; but $m$ in it gives $m + n, m + n' \in S$ with $(m+n') -
(m+n) = n' - n \neq 0$, so the Sidon condition determines the pair and hence $m$
uniquely, contradicting infinitude. Thus no member of $\Nstar + \Nstar$ contains $S$, and as
$\overline{S}$ is clopen, $\overline{S} \cap \Sigma_{2} = \emptyset$.
\end{proof}

This first-level mechanism is not new. Listed in increasing order, an infinite
Sidon set $\{x_{n}\}$ has pairwise distinct gaps, so $x_{n+1} - x_{n} \to \infty$,
and the lemma is a special case of \cite[Exercise~4.1.7]{HindmanStrauss1998}.
Sidon sets are also thin in $\Z$, so the lemma follows from the group results as
well \cite[Theorem~10]{FilaliLutsenkoProtasov2009} (see also \cite[\S 2 and
Theorem~2.2]{ProtasovProtasova2017}). The following corollary is likewise known:
for $(\N,+)$ it is \cite[Theorem~6.35]{HindmanStrauss1998} (see also the proof of
\cite[Theorem~21.22]{HindmanStrauss1998}), and it is the additive case of a
theorem of van Douwen \cite[Theorem~8.1]{vanDouwen1991}.

\begin{corollary}\label{cor:nowhere-dense}
$\Sigma_{2}$ is nowhere dense in $\Nstar$; equivalently, $\Nstar \setminus
\Sigma_{2}$ is dense in $\Nstar$.
\end{corollary}

\begin{proof}
Since $\Sigma_{2}$ is closed, it suffices to show $\Nstar \setminus \Sigma_{2}$
is dense, i.e.\ that every nonempty basic open set $\overline{A}$ with $A
\subseteq \N$ infinite contains a free ultrafilter outside $\Sigma_{2}$.  By a
classical greedy construction, every infinite $A \subseteq \N$ contains an
infinite Sidon set $B \subseteq A$: having chosen a finite Sidon subset
$\{a_{1}, \dots, a_{k}\} \subseteq A$, only finitely many values are excluded by
the Sidon condition, and $A$ is infinite, so a suitable $a_{k+1} \in A$ exists.
Since $B$ is infinite, $\{X \subseteq \N : B \setminus X \text{ is finite}\}$ is a
proper filter containing every cofinite set; any ultrafilter $p$ extending it
contains no finite set, so $p$ is free.  Then $B \in p$, so
$A \in p$ (as $B \subseteq A$), and $p \notin \Sigma_{2}$ by
Lemma~\ref{lem:firstgap}.  Thus $p \in \overline{A}
\setminus \Sigma_{2}$.
\end{proof}

\begin{lemma}\label{lem:rigid}
The sequence $e_{n} = 2^{2^{n}}$ satisfies:
\begin{enumerate}
\item[(i)] \textup{(Domination)} $e_{M} > e_{1} + \cdots + e_{M-1}$ for all $M
\geq 2$.
\item[(ii)] \textup{(Unique differences)} For each $\delta \neq 0$ there is at
most one pair $(a, c)$, $a < c$, with $e_{c} - e_{a} = \delta$.
\item[(iii)] \textup{(Unique sums of distinct indices)} If $\{i_{1}, \dots,
i_{r}\}$ and $\{i_{1}', \dots, i_{r}'\}$ are two sets of \emph{distinct} indices,
each of the same size $r$, with $e_{i_{1}} + \cdots + e_{i_{r}} = e_{i_{1}'} +
\cdots + e_{i_{r}'}$, then the two sets coincide.
\item[(iv)] Every infinite subsequence of $\{e_{n}\}$, and every translate of
one, is a Sidon set.
\end{enumerate}
\end{lemma}

\begin{proof}
(i) From $e_{n+1} = e_{n}^{2} \geq 2 e_{n}$ we get $\sum_{\ell<M} e_{\ell} < 2 e_{M-1}
\leq e_{M}$. (ii) The numbers $e_{c} - e_{a}$ ($a<c$) lie in the
intervals $[e_{c} - e_{c-1}, e_{c})$ by (i), and these are pairwise disjoint because
$e_{c+1} = e_{c}^{2} \geq 2e_{c}$; so $\delta$ determines $c$,
then $a$. (iii) Let $M$ be the largest index in either set; say it lies in the first. As
the indices on the first side are distinct, the terms other than $e_{M}$ are
$e$'s of strictly smaller distinct index, so their sum is at most $\sum_{\ell <
M} e_{\ell} < e_{M}$ by (i); hence the first sum lies in $[e_{M}, 2 e_{M})$. The
second sum equals it, so it too lies in $[e_{M}, 2 e_{M})$; since its terms are
distinct $e$'s, this is possible only if its largest term is $e_{M}$ as well.
Thus both sets contain $M$; remove it from each and induct on $r$. (iv) A
subsequence inherits unique differences from (ii), which is the Sidon property;
translation preserves differences.
\end{proof}

\section{The shift-intersection lemma}\label{sec:shift}

Fix a separated system $F_{1}, \dots, F_{j}$ with $F_{t} = \{ e_{n} + c_{t} : n
\in S_{t} \}$, the $S_{t}$ pairwise disjoint, and put $A = F_{1} + \cdots +
F_{j} \subseteq \Z$. Each element of $A$ is a sum $(e_{n_{1}} + c_{1}) + \cdots + (e_{n_{j}} +
c_{j})$ with $n_{t} \in S_{t}$; since the constants $c_{1} + \cdots + c_{j}$ are
fixed and, by disjointness, the $j$ indices are distinct,
Lemma~\ref{lem:rigid}(iii) determines the set $\{n_{1}, \dots, n_{j}\}$, and
disjointness of the $S_{t}$ then determines each $n_{t}$, so the representation
is unique.

\begin{lemma}\label{lem:shift}
Let $n \neq n'$ be positive integers and
\[
  I = \{ m \in \N : m + n \in A \text{ and } m + n' \in A \}.
\]
Then $I$ is contained in the union of a finite set and finitely many translates
of sub-sums $F_{t_{1}} + \cdots + F_{t_{i}}$ with $1 \leq i \leq j-1$; each such
sub-sum is again a separated-system sum.
\end{lemma}

\begin{proof}
Fix $\delta = n' - n \neq 0$. For $m \in I$ write $m + n = \sum_{t} (e_{n_{t}} +
c_{t})$ and $m + n' = \sum_{t} (e_{n_{t}'} + c_{t})$, with $n_{t}, n_{t}' \in
S_{t}$. Subtracting, the constants $c_{t}$ cancel and
\begin{equation}\label{eq:delta}
  \delta = \sum_{t=1}^{j} \bigl( e_{n_{t}'} - e_{n_{t}} \bigr).
\end{equation}
Because the representations respect the system coordinate-by-coordinate, every
index appearing in the $t$-th coordinate, primed or not, lies in $S_{t}$, and
the $S_{t}$ are disjoint.

Let $M$ be the largest index occurring in \eqref{eq:delta}. The index $M$ lies in
exactly one $S_{t}$ (disjointness), so $e_{M}$ can occur in \eqref{eq:delta} only
within coordinate $t$, as $e_{n_{t}'}$ (if $n_{t}' = M$) or $e_{n_{t}}$ (if
$n_{t} = M$). Suppose $e_{M}$ does not cancel there, i.e.\ not both $n_{t}' =
n_{t} = M$; then coordinate $t$ contributes $\pm e_{M}$ plus a term of smaller
index; in particular $n_{t} \neq n_{t}'$, so $M \geq 2$. The remaining $2j - 1$ terms of \eqref{eq:delta} each have index $< M$,
hence are at most $e_{M-1}$, so they total at most $(2j-1)\,e_{M-1}$ in absolute
value. Thus
\[
  |\delta| \;\geq\; e_{M} - (2j-1)\,e_{M-1}
  \;=\; e_{M-1}\bigl(e_{M-1} - (2j-1)\bigr),
\]
using $e_{M} = e_{M-1}^{2}$. For fixed $j$ the right-hand side tends to infinity
with $M$, so there is $B = B(\delta, j)$ such that it exceeds $|\delta|$ for every
$M > B$. Hence if $M > B$, the term $e_{M}$ cancels, which forces $n_{t}' = n_{t} =
M$; coordinate $t$ then contributes $0$ to \eqref{eq:delta} and its common index
is unconstrained by the equation.

Deleting a coordinate whose indices cancel leaves an equation of the same shape
with one fewer coordinate, to which the argument applies again; with $c \leq j$
coordinates the remaining terms total at most $(2c-1)\,e_{M-1} \leq
(2j-1)\,e_{M-1}$, so the same $B$ works. Iterating,
each step either removes a coordinate whose largest index cancels or terminates.
Not every coordinate can be removed, since $\delta \neq 0$; so the process
terminates at a largest surviving index $M^{*}$ that does not cancel, and then
$M^{*} \leq B$.

Thus each $m \in I$ determines a set $D(m) \subseteq \{1, \dots, j\}$ of
coordinates, the ones surviving when the process terminates, with three
properties: $D(m) \neq \emptyset$, since $\delta \neq 0$; $n_{t} \leq B$ for $t
\in D(m)$; and $n_{t} = n_{t}'$ for $t \notin D(m)$. There are only finitely
many possibilities for $D(m)$ together with the values $n_{t}$, $t \in D(m)$.
For each of them, the corresponding elements $m = \sum_{t}(e_{n_{t}} + c_{t}) - n$ lie in
the translate
\[
  \Bigl(\sum_{t \in D}(e_{n_{t}} + c_{t}) - n\Bigr) + \sum_{t \notin D} F_{t}
\]
of the sub-sum over the $j - |D| \leq j - 1$ remaining coordinates. When $D =
\{1, \dots, j\}$ this translate is a single point, and we put it into the finite
set. This proves the lemma.
\end{proof}

\section{Proof of the Master Lemma}\label{sec:master}

\begin{proof}[Proof of Theorem~\ref{thm:master}]
We argue by induction on $j$, proving
\[
  P(j):\quad \begin{array}{l}\text{for every separated system } F_{1}, \dots,
  F_{j} \text{ and every } s \in \Nstar,\\ \text{if } (F_{1} + \cdots + F_{j})
  \cap \N \in s \text{ then } s \notin \Sigma_{j+1}.\end{array}
\]

\emph{Base $j = 1$.} A single $F_{1}$ is an integer translate of an infinite
subsequence of $\{e_{n}\}$, hence a Sidon set by Lemma~\ref{lem:rigid}(iv), and so is its infinite subset
$F_{1} \cap \N$. If
$F_{1} \cap \N \in s$ then $s \notin \Sigma_{2}$ by Lemma~\ref{lem:firstgap}.

\emph{Inductive step.} Let $j \geq 2$ and assume $P(1), \dots, P(j-1)$. Let
$F_{1}, \dots, F_{j}$ be a separated system, $A = (F_{1} + \cdots + F_{j}) \cap \N \in s$,
and suppose for contradiction that $s \in \Sigma_{j+1}$. By the cumulative
membership criterion (Lemma~\ref{lem:membership}) applied to the member $A \in
s$, there are $r \in \Nstar$ and $v \in \Sigma_{j}$ with $A \in r + v$. Then
\[
  T = \{ n : -n + A \in v \} \in r
\]
is infinite, $r$ being free. Choose distinct $n, n' \in T$; then $-n + A, -n' + A
\in v$, so $I = (-n+A) \cap (-n'+A) \in v$; this is the set $I$ of Lemma~\ref{lem:shift}
(formed with $F_{1} + \cdots + F_{j} \subseteq \Z$), since $m + n \in \N$ for $m \in \N$. By Lemma~\ref{lem:shift}, $I$ is
contained in the union of a finite set and finitely many translates of sub-sums,
each with between $1$ and $j-1$ summands. As $v \in \Sigma_{j} \subseteq \Nstar$ is
free, the finite part is not in $v$. If there are no such translates, then $I$ is finite; but $I \in v$ and
$v$ is free, a contradiction. Otherwise, since $v$ is an ultrafilter, one of the
translates, call it $Y$, satisfies $Y \cap \N \in v$. Thus $Y$ is a
translate of a sub-sum $F_{t_{1}} + \cdots + F_{t_{i}}$ with $1 \leq i \leq
j-1$, and the inductive hypothesis $P(i)$ is available.

Now $Y$ is a translate of the separated-system sum $F_{t_{1}} + \cdots +
F_{t_{i}}$. Adding the translation constant to one of the summands $F_{t_{1}}$
keeps it an integer-translate of an infinite subsequence of $\{e_{n}\}$ and
leaves the index-disjointness intact, so $Y$ is again a separated-system sum,
of size $i \leq j-1$. By the inductive hypothesis $P(i)$, since $Y \cap \N
\in v$ we get $v \notin \Sigma_{i+1}$. As $i + 1 \leq j$, we have $\Sigma_{j}
\subseteq \Sigma_{i+1}$ (Lemma~\ref{lem:descending}), and therefore $v \notin \Sigma_{j}$. This contradicts $v \in
\Sigma_{j}$. Hence $s \notin \Sigma_{j+1}$, proving $P(j)$.
\end{proof}

\begin{remark}
The induction never uses that the peeled-off remainder $v$ is itself an iterated
sum; it uses only that $v \in \Sigma_{j}$ and $v$ is free. This is precisely why the
argument runs inside the cumulative filtration, whose levels are closures: the
obstruction statement $P(i)$ speaks of an arbitrary free ultrafilter containing a
separated-system sum, and applies to a limit point exactly as to an actual sum.
\end{remark}

\begin{proof}[Proof of Theorem~\ref{thm:main}]
As $p$ is a $k$-fold sum of free ultrafilters, $p \in \Sigma_{k}$ by
Proposition~\ref{prop:two-sided-k}. By Lemma~\ref{lem:summembers}, $A_{k} =
E_{0}^{(k)} + \cdots + E_{k-1}^{(k)} \in q_{0} + \cdots + q_{k-1} = p$. The sets $E_{0}^{(k)},
\dots, E_{k-1}^{(k)}$ have disjoint index sets (the residue classes modulo $k$),
so they form a separated system of size $k$; by Theorem~\ref{thm:master} applied
to $s = p$, $A_{k} \in p$ gives $p \notin \Sigma_{k+1}$. Hence $p \in \Sigma_{k}
\setminus \Sigma_{k+1}$.
\end{proof}

\section{Remarks}\label{sec:remarks}

\subsection*{The pure filtration}
Let $\Pi_{k} = \overline{\Nstar + \cdots + \Nstar}$ ($k$ summands) be the pure
filtration, the closure of the $k$-fold sums in a single step. Its group
counterpart $\overline{(G^{*})^{n}}$ is also considered by Protasov and Protasova,
who state that it is strictly decreasing, remarking that this can be proved by
analogy with their Theorem~4.2 \cite[Theorem~4.3]{ProtasovProtasova2017}. Here
$\Pi_{k} \subseteq \Sigma_{k}$, since $\Sigma_{k}$ is closed and contains the
$k$-fold sums (Proposition~\ref{prop:two-sided-k}). The two agree for $k \leq 2$; whether $\Pi_{k} = \Sigma_{k}$ for
$k \geq 3$ we do not know (Question~\ref{q:pure}). Our
witnesses lie in the pure gaps as well: since $p$ is a $k$-fold sum, $p \in
\Nstar + \cdots + \Nstar \subseteq \Pi_{k}$, while $p \notin \Sigma_{k+1}
\supseteq \Pi_{k+1}$, so $p \in \Pi_{k} \setminus \Pi_{k+1}$. The point of the present argument
is that the obstruction never required the pure structure: the induction runs in
the cumulative filtration directly, peeling one summand at a time and using only
that the remainder lies in $\Sigma_{j}$.

\subsection*{Scope of the rigidity}
The proof uses the properties of Lemma~\ref{lem:rigid} together with the growth
estimate $e_{M} - (2j-1)e_{M-1} \to \infty$ from the proof of
Lemma~\ref{lem:shift}. Any sequence with properties (i)--(iv) and $e_{n+1}/e_{n}
\to \infty$ would serve in place of $2^{2^{n}}$; the doubly exponential choice is
merely the most transparent. In
particular the partition into residue classes is one convenient way to produce a
separated system, but the Master Lemma holds for every separated system, so many
other explicit sets witness each level.

\subsection*{The group \texorpdfstring{$\Z$}{Z}}
Apart from intersecting with $\N$, nothing in Sections~\ref{sec:shift}
and~\ref{sec:master} uses that $n$, $n'$ or the elements of $I$ are positive. Running the same argument in $\beta\Z$, with
the chain $I_{\Z,0} = \Z^{*}$, $I_{\Z,k} = \overline{\Z^{*} + I_{\Z,k-1}}$ of
Protasov and Protasova, shows that no ultrafilter in $I_{\Z,j}$ contains a
separated-system sum $F_{1} + \cdots + F_{j}$; Lemmas~\ref{lem:descending},
\ref{lem:membership}, \ref{lem:summembers} and~\ref{lem:firstgap} and
Proposition~\ref{prop:two-sided-k} hold in $\beta\Z$ with the same proofs. Since $\bN$ is a closed subsemigroup of $\beta\Z$, and for $p, q \in \bN$
and $A \subseteq \Z$ we have $A \in p + q$ in $\beta\Z$ exactly when $A \cap \N
\in p + q$ in $\bN$, the witnesses of Theorem~\ref{thm:main} lie in
$I_{\Z,k-1} \setminus I_{\Z,k}$. They therefore lie in every gap of the
Protasov--Protasova chain for $\Z$ as well, witnessed by the same explicit sets
$A_{k}$, and, being $k$-fold sums, in the gaps of
the pure chain $\overline{(\Z^{*})^{k}}$, whose strictness Protasov and Protasova
assert by analogy with their Theorem~4.2 \cite[Theorem~4.3]{ProtasovProtasova2017}.

\section{The bottom of the filtration}\label{sec:bottom}

Proposition~\ref{prop:two-sided} gives $\overline{\K} \cup \overline{\E} \subseteq
\Sigma_{\infty}$.  The following rules out the identification $\Sigma_{\infty} =
\overline{\K}$.

\begin{proposition}\label{prop:strict-K}
$\overline{\K} \subsetneq \Sigma_{\infty}$.
\end{proposition}

\begin{proof}
Hindman and Strauss~\cite[Theorem~3.5]{HindmanStrauss2011} construct an idempotent
$e \in \E$ outside $\overline{\K}$: they pick $e \in \overline{\FS(\langle
2^{2n} \rangle_{n=1}^{\infty})}$ \cite[Lemma~5.11]{HindmanStrauss1998} and
observe that $\FS(\langle 2^{2n} \rangle)$ is not piecewise
syndetic~\cite[Corollary~4.2]{AdamsHindmanStrauss2008},
so $e \notin \overline{\K}$ by~\cite[Corollary~4.41]{HindmanStrauss1998}.  By
Proposition~\ref{prop:two-sided}, $e \in \overline{\E} \subseteq \Sigma_{\infty}$.
Hence $e \in \Sigma_{\infty} \setminus \overline{\K}$.  Essentially the same set
(indexed from $n = 0$) appears in
\cite[Exercise~4.4.3]{HindmanStrauss1998}, and the argument goes back to the proof
of \cite[Corollary~6.33]{HindmanStrauss1998}.
\end{proof}

Proposition~\ref{prop:strict-K} rules out $\Sigma_{\infty} = \overline{\K}$.
The next natural candidate is the minimal closed two-sided ideal of $\Nstar$
containing all idempotents; call it $M$.  For a group $G$, Protasov and
Protasova~\cite[Theorem~4.1]{ProtasovProtasova2017} characterized the analogous
ideal of $\beta G$ as $\mathrm{Sc}^{\wedge}_{G}$, the dual (in their
correspondence between ideals in $\mathcal{P}(G)$ and closed subsets of $\beta
G$) of the family of \emph{scattered} subsets of $G$.

\begin{question}\label{q:identification}
Does $\Sigma_{\infty} = M$?  A positive answer would identify
$\Sigma_{\infty}$ with the minimal closed two-sided ideal containing the
idempotents.  A negative answer would show that the depth filtration bottoms out
strictly above $M$.  We have $M \subseteq \Sigma_{\infty}$ from
Proposition~\ref{prop:two-sided} (for groups, compare
\cite[Theorem~4.2(ii)]{ProtasovProtasova2017}); the question is whether the reverse inclusion
holds.
\end{question}

For groups, Protasov and Protasova conjecture that $\mathrm{Sc}^{\wedge}_{G}
\neq \bigcap_{n \in \omega} \overline{(G^{*})^{n}}$
\cite[Conjecture~4.1]{ProtasovProtasova2017}, where $\mathrm{Sc}^{\wedge}_{G}
\subseteq \overline{(G^{*})^{n}}$ for every $n$
\cite[Theorem~4.3(ii)]{ProtasovProtasova2017}. The analogue for $\N$, $M
\subsetneq \bigcap_{k} \Pi_{k}$, would give a negative answer to
Question~\ref{q:identification}, since $\Pi_{k} \subseteq \Sigma_{k}$.

\begin{question}\label{q:pure}
Is $\Pi_{k} = \Sigma_{k}$ for every $k \geq 3$?
\end{question}

\section*{Acknowledgments}

The author thanks Boris \v{S}obot for reading an earlier version of this paper
and for helpful correspondence.

\end{document}